\documentclass[12pt]{amsart}
\usepackage{lineno,hyperref}
\modulolinenumbers[5]
\usepackage{comment}
\usepackage[english]{babel}
\usepackage[utf8]{inputenc}
\usepackage{amsmath}
\usepackage{shuffle}
 \usepackage{graphicx}
\usepackage[colorinlistoftodos]{todonotes}
\usepackage{amsthm}
\usepackage{amsfonts}
\usepackage{upgreek}
\usepackage[ruled,vlined]{algorithm2e}
\usepackage{float}
\usepackage{color}
\usepackage{xcolor}
\usepackage[all]{xy}
\usepackage{subfig}
\usepackage{tabularx}
\usepackage{amssymb}
\usepackage{resizegather}
\usepackage{color}
\usepackage{amsmath}
\usepackage{tikz-cd}
\usepackage{qtree}
\usepackage{graphicx}
\usepackage{url}

\usetikzlibrary{matrix, calc, arrows}
\usepackage{tikz}
\usetikzlibrary{positioning}
\usepackage{tkz-euclide}
\usepackage[foot]{amsaddr}
\usepackage{fullpage}
\usepackage{array}
\usepackage{mathabx}

\usepackage{adjustbox}
\newtheorem{theorem}{Theorem}[section]
\newtheorem*{theorem*}{Theorem}
\newtheorem{lemma}[theorem]{Lemma}
\newtheorem*{lemma*}{Lemma}

\newtheorem{corollary}[theorem]{Corollary}
\newtheorem{entry}[theorem]{Entry}
\newtheorem*{corollary*}{Corollary}
\newtheorem{proposition}[theorem]{Proposition}

\theoremstyle{definition}

\newtheorem{example}[theorem]{Example}
\newtheorem{remark}[theorem]{Remark}

\newtheorem*{exercise*}{Exercise}

\usepackage{mathtools}
\def\multiset#1#2{\ensuremath{\left(\kern-.3em\left(\genfrac{}{}{0pt}{}{#1}{#2}\right)\kern-.3em\right)}}

\newcommand{\chain}[1][n]{\langle #1\rangle}

\newcommand{\nsomega}[2][n]{\Omega(#2,#1)}
\newcommand{\somega}[2][n]{\Omega^{\circ}(#2,#1)}

\newcommand{\zetaf}[1][X]{\mathfrak{Z}\left(#1\right)}
\newcommand{\zetafp}[1][X]{\mathfrak{Z}^+\left(#1\right)}

\newcommand{\zetafn}[1][X]{\mathfrak{Z}_{\mathbb{N}}^+\left(#1\right)}
\newcommand{\zetafpn}[1][X]{\mathfrak{Z}_{\mathbb{N}}\left(#1\right)}

\usetikzlibrary{decorations.markings}
\tikzstyle{vertex}=[circle, draw, inner sep=0pt, minimum size=6pt]

\title{A poset version of Ramanujan results on Eulerian numbers and zeta values}

\author[1]{{Eric R}
 {Dolores-Cuenca}$^1$}\email{$^1$eric.rubiel@yonsei.ac.kr}\address{$^1$Department of mathematics, {Yonsei University}, { {Seoul}, {Korea}}}

\author[2]{{Jose L. Mendoza-Cortes$^2$}}\email{$^2$jmendoza@msu.edu}\address{$^2$Department of Chemical
Engineering \& Materials
Science, Michigan State
University, East Lansing, MI
48824, USA}

\begin{document}

\maketitle

\begin{abstract} 
We explore the operad of finite posets and its algebras. We use order polytopes to investigate the combinatorial properties of zeta values. By generalizing a family of zeta value identities, we demonstrate the applicability of this approach. In addition, we offer new proofs of some of Ramanujan's results on the properties of Eulerian numbers, interpreting his work as dealing with series inheriting the algebraic structure of disjoint unions of points. Finally, we establish a connection between our findings and the linear independence of zeta values.
\end{abstract}

\section{Introduction}

During the final week of December 2021, the Mathematics Department at Yonsei University organized a screening of math-related films. Among them was ``The Man Who Knew Infinity'', a biopic about the life of Ramanujan. Following the screening, we headed to the library to peruse the copy \cite{R} of ``Ramanujan's Notebooks''.

Ramanujan's notebooks do not feature his original proofs. As per \cite{R}, there are various plausible reasons for this omission. Firstly, the notebooks were intended to be a compilation of his findings rather than a comprehensive guide. Secondly, given Ramanujan's financial struggles, the exorbitant cost of paper may have played a role. Additionally, it is feasible that Ramanujan emulated the style of Carr's book \cite{carr}, which presents theorems with only rough sketches of proofs. In an interview \cite{conv}, S. Janaki, Ramanujan's widow, revealed that her husband ``was fearful that English mathematicians would steal his mathematical secrets while in England''.

Chapter 5 of Ramanujan's first notebook left us astounded as we stumbled upon the following declaration:

\begin{entry}[{\cite[Chapter 5, Corollary 6]{R}}
]\label{Rm} Fix $n\in \mathbb{N}$. There are $\{A_i\}_{1\leq i\leq n}\in\mathbb{N}$ such that for all $1\leq r\leq n$:
\begin{eqnarray*}
\sum_{k=1}^r A_k{r\choose k}&=&r^n.
\end{eqnarray*}
\end{entry}

The aforementioned theorem can be derived directly from the theory proposed in our paper~\cite{posets}. In the formula, $r^n$ represents the number of ways to label the disjoint union of $n$ points using numbers between $1$ and $r$. The coefficient $A_k$ counts a particular set of $k$--simplices in a given triangulation of the $n$--cube. As we progressed through Chapter 5, we stumbled upon several additional outcomes that align with our perspective.

\subsubsection{What we do}

We introduce the action of an operad on zeta values. An operad of sets comprises a set of $n$-ary operations $O(n)$, with rules for composition and actions of the symmetric group. These rules of composition represent an abstraction of the composition of functions, while the action of the symmetric groups reflects the permutations of inputs. Consequently, an algebra over an operad $O$ is a set $X$ in which every element of $O(n)$ defines an operation $X^{n}\rightarrow X$. For precise definitions, refer to \cite{ope, atf}.

Consider an $n$-ary poset $P=\{x_1,\cdots,x_{n}, <_P\}$, which we interpret as an operation, and let $P_1,\ldots,P_n$ be posets. We define the action of $P$ on $P_1,\ldots,P_n$ to be the lexicographic sum of $P_1,\ldots,P_n$ over $P$. This means that we replace each vertex $x_i$  with the entire poset $P_i$, and then take the resulting poset with the induced ordering. In other words, the operad of finite posets acts on the set of finite posets.

The language of operads enables us to study algebraic properties that are inherent to objects that have a defined composition, even if these objects do not satisfy all the criteria of traditional algebraic structures like groups or rings.

Our main result is the proof of the following identity. Let $\zeta(k)=\sum_{n=1}^\infty \frac{1}{n^k}$ denote the Riemann zeta function. Given $k\in\mathbb{N}$, \begin{eqnarray} \sum_{n=k}^\infty (-1)^{n+1} {n\choose k}(\zeta(n+1)-1)&=&(-1)^{k+1}(\zeta(k+1)-1-1/2^{k+1}). \label{fEq:1}\end{eqnarray} This identity relates the alternating sum of a particular combination of binomial coefficients and zeta values to a closed-form expression involving the Riemann zeta function evaluated at integer arguments. Notably, when we attempted to compute the left-hand side of Equation~\eqref{fEq:1} using a variety of public and private software packages in January 2023, we were unable to obtain any output.

In the field of combinatorics, the order polynomial introduced by Stanley is a polynomial that is associated with a finite poset. It is well known that  binomial polynomials serve as a basis for order polynomials. Our work shows that order polynomials inherit the action of the operad of finite posets, and that Equation~\eqref{fEq:1} establishes a relationship between order polynomials and zeta values as algebras over the operad of posets. This allows us to induce an action of the operad of finite posets on zeta values. 

We establish a connection between the theory of order polytopes and zeta values, allowing us to compute results such as the following example: By considering the cube as the union of 6 pyramids $\Delta[3]$, we can use the inclusion-exclusion principle to account for the faces resulting from the intersection of two pyramids. Thus, we have $[0,1]^3=\sqcup_6 \Delta[3]\setminus (\sqcup_6\Delta[2])\sqcup \Delta[1]$, where the final term corresponds to the diagonal shared by all triangles $\Delta[2]$ and all pyramids $\Delta[3]$. From this geometric interpretation, we derive the identity
\[\sum_{k=1}^\infty \frac{6{k\choose 3}-6{k\choose 2}+{k\choose 1}}{x^{k+1}}=\sum_{k=1}^\infty   (-1)^{k+1}\frac{k^3}{x^{k+1}}=\frac{6}{(1-x)^{3+1}}-\frac{6}{(1-x)^{2+1}}+\frac{1}{(1-x)^{1+1}}, \]
which we express in finite form as
\begin{eqnarray*}
\sum_{k=1}^\infty   (-1)^{k+1}k^3(\zeta(k+1)-1)&=&{6}({\zeta(3+1)-1-\frac{1}{2^{3+1}}})-{6}({\zeta(2+1)-1-\frac{1}{2^{2+1}}})\\
&&+({\zeta(1+1)-1-\frac{1}{2^{1+1}}}).
\end{eqnarray*}
This illustrates the power of our approach in connecting seemingly disparate areas of mathematics.
 
We demonstrate a novel technique for generating identities of zeta values from specific triangulations of polytopes, establishing a connection between enumerative combinatorics and number theory. Our approach is illustrated through a detailed analysis of order polytopes, and we establish the validity of this assignment for this class of polytopes. Notably, our work expands upon the insights of Ramanujan. From our point of view, Ramanujan studied identities linked to $n$-cubes.
 
Once we identified generators in an algebra over an operad, the action of the operad can be described in the aforementioned basis via structural constants. The action of the operad of posets on order polytopes is simple to describe, we then compute the structural constants on polytopes and use those constants to describe the action of the operad on zeta values. 
Our theorems yield several results that are reminiscent of those discovered by Ramanujan. Specifically, when we apply our theorems to the disjoint union of $n$ points $P=\sqcup_n \chain[1]$, we recover some of Ramanujan's statements, which were previously proven in the book \cite{R}. Our proofs are distinct from those given in \cite{R} and, in some cases, shorter. To facilitate a comparison with Ramanujan's original work, we present his statements using the ``Entry'' environment, as shown in Entry~\ref{Rm}.

This paper is structured as follows: In Section~\ref{Sec:OS}, we delve into combinatorics. We begin by presenting Lemma~\ref{IEP}, which provides an interpretation of the basis vectors $\{{x\choose i}\}_{i\in\mathbb{N}}$ in terms of counting certain simplices in the canonical triangulation of the order polytope of $P$, this result generalizes~\cite[Proposition 3.6]{posets}. Next, in Theorem~\ref{thm:main}, we define the action of the operad of finite posets on Stanley order polynomials. Additionally, we introduce the concept of order series, a family of generating functions associated with a poset, and show that they possess the structure of an algebra over the operad of finite posets, extending the results of \cite[Proposition 2.3]{posets}. We also explore the poset $\{x<y>z<w\}$, which introduces a 4th-ary associative, non-commutative operation. A series parallel poset is generated by the action of the operations $\{x<y\}$, $\{x,y\}$ on chains $\chain[n]=\{1<2<\cdots<n\}$. Our work highlights that the action of $\{x<y>z<w\}$ is distinct from that of any poset generated by the operations $\{x<y\}$, $\{x,y\}$, which expands upon the results of \cite{posets}. 
We provide detailed explanations of our findings and include Lemma~\ref{Lemma:4} to support our claims.

In Section~\ref{sub:union}, we investigate posets that are formed by the disjoint union of points. Our analysis leads to new proofs for certain results originally presented by Ramanujan. Notably, our version of Corollary~\ref{Cor:up} is slightly more general than the one found in \cite{R}.

In Section~\ref{sec:Poset}, we demonstrate the relationship between {\cite[Chapter 5, Example 1,2.]{R}} and \cite[Chapter 5, Collolary 6 iv]{R} with the disjoint union of points. To extend these results to finite posets, we present Lemma~\ref{lemma:comp} and Theorem~\ref{theorem:R2}. Notably, Theorem~\ref{theorem:R2} yields a shorter proof of~\cite[{Chapter 5, Corollary 6. iv)}] {R} when the poset is limited to be the disjoint union of points. Our paper's primary contribution is Theorem~\ref{theorem:R2}, which establishes Equation~\eqref{fEq:1}. Moreover, the series of the form $\sum_{i=1}^\infty a_i(\zeta(i+1)-1), a_i\in\mathbb{Q}$, are commonly known as rational zeta series~\cite{compuzeta}. Our research expands the scope of rational zeta series that we can compute.

In Section~\ref{Sec:spz}, we explore two algebraic structures of order polynomials/order series: the vector space structure over $\mathbb{Z}$ and the algebraic structure over the operad of posets. In Corollary~\ref{CorZeta}, we demonstrate how the vector space structure allows us to define the action of the operad of posets on objects that are parameterized by binomial polynomials, such as zeta values. Furthermore, we introduce a function $\tilde{n}$ that maps ${x\choose k}$ to the value of Equation~\eqref{fEq:1}. We prove in Theorem~\ref{thm:li} that the injectivity of $\tilde{n}$ is equivalent to the linear independence of zeta values, which is a prominent unsolved problem in number theory.

Section~\ref{Sec:pa} delves into the potential connections between posets combinatorics and some of the statements presented in Ramanujan's notebooks. While no explicit connection has been established previously, we discuss which statements are likely to be explained with our approach.

\section{Order series}\label{Sec:OS}
 
In this section, we venture into the realm of enumerative combinatorics, building upon and generalizing the primary results of \cite{posets}. Throughout our exploration, we will illustrate our approach with explicit poset examples.

For our analysis, we focus solely on finite posets and denote the poset $1<2<\cdots<n$ as $\chain[n]$. Furthermore, we use $|P|$ to denote the number of elements in a given poset $P$.

In his work \cite{beginning}, Stanley introduced the weak order polynomials $\nsomega[n]{P}$ and the strict order polynomials $\somega[n]{P}$. Specifically, $\nsomega[n]{P}$ counts the number of non-strict order-preserving maps $P\rightarrow \chain[n]$ such that $x<y$ implies $f(x)\leq f(y)$, while $\somega[n]{P}$ counts the number of strict order-preserving maps $P\rightarrow \chain[n]$ such that $x<y$ implies $f(x)< f(y)$.

In our analysis, we will utilize the following renowned result established by Stanley:
\begin{theorem}[\cite{crt, enumerative}]\label{srt}
Let $P$ be a poset. Then, we have
\[(-1)^{|P|}\somega[-x]{P}=\nsomega[x]{P}.\]
This result is known as Stanley reciprocity. 
\end{theorem}

\begin{remark}
In order to establish a relationship between $\somega[x]{P}$ and $\nsomega[x]{P}$ using Stanley reciprocity, we must first compute the coefficient $(-1)^{|P|}$, which is dependent on the specific poset $P$. Therefore, simply knowing the number of strict maps $hom_{\hbox{strict}}(P,n)$ is insufficient for determining the number of weak maps $hom_{\hbox{weak}}(P,n)$; additional information about the poset $P$ that induced the polynomials is required.
\end{remark}

Below are some important properties of binomial coefficients that will be utilized throughout the paper:
\begin{itemize}
\item The binomial coefficient ${n\choose k}$ represents the number of ways to choose $k$ items from a set of $n$ distinct objects.
\item If we expand the simplex $\Delta[k]$ $n$ times in each direction to form $n\Delta[k]$, the number of lattice points in the simplex is ${n\choose k}$.
\item For a finite chain $\chain[k]$, we have $\somega[n]{\chain[k]}={n\choose k}$, since the image of the order-preserving map can be obtained by choosing $k$ different values from $n$.
\item The $n$-th coefficient of the generating function $\frac{x^k}{(1-x)^{k+1}}$ is ${n\choose k}$. This result can be derived from~\cite[Equation~(1.3)]{lcomb} or~\cite[Equation~(1.5.5)]{gf}.
\end{itemize}

Given a poset $P$, its order polytope $Poly(P)$ is a fundamental object in order theory and combinatorial geometry~\cite{two}. Geometrically, we can associate each element in $P$ with an axis in $[0,1]^{|P|}$ and define $Poly(P)$ as the subspace of $[0,1]^{|P|}$ determined by the poset inequalities. Specifically, $Poly(P)$ is the set of functions $f:P\rightarrow [0,1]$ that satisfy $f(x)\leq f(y)$ whenever $x\leq_P y$. The canonical triangulation of $Poly(P)$ induced by the poset structure decomposes $Poly(P)$ into $d_{|P|}$ simplices of dimension $|P|$ that are glued together along $d_{|P|-1}$-simplices of dimension $|P|-1$. However, we need to remove the extra copies of $|P|-1$-simplices that were counted twice. For example, when $|P|=2$, we have $Poly(P)=\Delta[2]\sqcup\Delta[2]\setminus \Delta[1]$. More generally, we can write:
\[Poly(P)=\sqcup_{d_{P,|P|}} \Delta[|P|]\setminus( \sqcup_{d_{P,|P|-1}} \Delta[|P|-1])\sqcup ( \sqcup_{d_{P,|P|-2}} \Delta[|P|-2])\setminus \cdots (\sqcup_{d_{P, 1}} \Delta[1]).\]
We use this decomposition to define the coefficients  $d_{P,i}$.

The strict order polynomials are integer-valued polynomials, and we can obtain their Gregory-Newton decomposition~\cite{IVP,AMS} as $\Omega^\circ(P,x)=\sum_{i=0}^{|P|} a_i {x\choose i}$. The non-negativity and integrality of the coefficients $a_i$ have been proven in~\cite{crts} by characterizing them as the number of surjective maps from $P$ into the chain $\chain[i]$. In geometric terms, this quantity corresponds to the number $d_{P,i}$ in the canonical triangulation of the order polytope $Poly(P)$~\cite{two}. This leads to an insightful interpretation of the binomial basis vectors, which can be compared to the well-known $h^*$ vectors\cite{monotonicity} and $f^*$ vectors~\cite{fvect}:

\begin{lemma}[Inclusion Exclusion]\label{IEP} Let $P$ be a finite poset, $n=|P|$. If the canonical triangulation of the order polytope of $P$ is given by
 \[Poly(P)=\sqcup_{d_{P,|P|}} \Delta[|P|]\setminus( \sqcup_{d_{P,|P|-1}} \Delta[|P|-1])\sqcup ( \sqcup_{d_{P,|P|-2}} \Delta[|P|-2])\setminus \cdots (\sqcup_{d_{P, 1}} \Delta[1]),\]
then, the strict order polynomial and the weak order polynomial of $P$ satisfy $\somega[x]{P}=\sum_{i=j_0}^n d_{P,i}{x\choose i}$ and $\nsomega[x]{P}=\sum_{i=1}^n (-1)^{|P|-i}d_{P,i}\multiset{x}{i}$.  \end{lemma}
   
The multiset $\multiset{x}{n}={x+n-1\choose n}$ is utilized in deriving the representation of $\nsomega[x]{P}$ from that of $\somega[x]{P}$, through the application of Stanley reciprocity. One may question the relevance of the multiset, which can be justified as follows: consider a set of positive points denoted as $\{a,b,\cdots, l|\}$, and a set of negative points denoted as $\{|u,v,\cdots, z\}$. The reason for the appearance of the multiset in the enumeration of poset maps is that weak order preserving maps $\chain[k]\rightarrow \chain[n]$ can be seen as subsets of size $k$ from a set with $n$ negative elements $\{|1,2,\cdots, n\}$. For further details, please refer to~\cite{hybrid, nbinomial}, and~\cite{negative} which provide a new interpretation of Theorem~\ref{srt}. It is worth noting that negative sets will not be used in the remainder of this paper.

In collaboration with Ayoub Youssefi Abqari (private communication on 10/31/2021), we have discovered that the results presented in Corollary 3, Corollary 4, Theorem 3, and Theorem 4 of \cite{sums}, which deal with iterated sums, can all be derived through poset arguments. For instance, consider Corollary 4 of \cite{sums}:
\begin{corollary*}[ 4  of \cite{sums}]
     \[{n+k-q\choose k}=\sum_{i_k=q}^n\sum_{i_{k-1}=q}^{i_{k}}\cdots\sum_{i_2=q}^{i_3}\sum_{i_1=q}^{i_2} 1.\]
  \end{corollary*}
\begin{proof}
We may represent the potential values of a map from $\chain[k]$ to $\chain[n]$ as $i_1,\cdots,i_k$. The right-hand side of the equation is counting the number of order-preserving maps from $\chain[k]$ to $\chain[n]$ where every value is greater than or equal to $q$. However, we can instead consider all non-strict maps from $\chain[k]$ to $\chain[n-(q-1)]$, which amounts to $({n-q+1\choose k})={n+k-q\choose k}$.
\end{proof}

After obtaining the vectors of $\somega[x]{P}$, we utilize Lemma~\ref{IEP} to extract geometric information about the order polytope of $P$. As an illustration, let us consider the poset $P=\{x<y>z<w\}$. The calculation of $\somega[x]{P}$ yields ${x\choose 2}+5{x\choose 3}+5{x\choose 4}$. By applying Lemma~\ref{IEP}, we conclude that $Poly(P)$ can be constructed by assembling 5 copies of the simplex $\Delta[4]$, where each simplex is attached to two other $\Delta[4]$ along five $\Delta[3]$ simplices in such a way that all $\Delta[4]$ simplex share a $\Delta[2]$ triangle.

Operads~\cite{ope, atf} provide a powerful tool to handle customized operations. Essentially, an operad of sets comprises a collection of $n-$ary operations $O(n)$, governed by composition rules and acted upon by the symmetric group. For instance, given a set $X$, there exists an operad $End_X$, where $End_X(n)$ denotes the set of endomorphisms $X^n\rightarrow X$. In this context, an algebra over an operad $O$ refers to an operadic morphism $\varphi_A:O\rightarrow End_A$.

Consider a finite poset $P=\{x_1,\cdots,x_{|P|},\leq_{P}\}$. To define $P(P_1, \cdots, P_{|P|})$, the Lexicographic sum~\cite{osets, order} is employed. This operation involves taking the set $\sqcup P_i$ and imposing the order $a\leq_{P(P_1, \cdots, P_{|P|})} b$ if either $a\leq_{P_i} b$ for some $i$, or if $a\in P_r$, $b\in P_t$, and $x_r\leq_P x_t$. For instance, $\{x<y\}(\{a,b\},\{c,d\})=\{a<c,a<d,b<c, b<d\}$, as depicted in Figure~\ref{fig:my_label}. We denote the operad of finite posets by $\mathcal{FP}$.

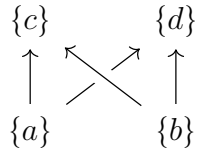
\begin{figure}[htb]
    \centering
\begin{tikzcd}
\{c \} & \{d\} \\
\{a\}\uar \urar&\{ b\}\ular[crossing over]\uar 
  \end{tikzcd} %
 \caption{We present the Lexicographic sum $\{x<y\}(\{a,b\},\{c,d\})$ of two posets $\{a,b\}$ and $\{c,d\}$, with the original poset $\{x<y\}$.  An arrow from $r$ to $s$ implies that $r<s$. The resulting poset is constructed using the lexicographic order, as defined in the main text.} 
    \label{fig:my_label}
\end{figure}

\begin{theorem}\label{thm:main}
The operad of finite posets acts on the Stanley order polynomials.
\end{theorem}
\begin{proof}
To prove the theorem, we need to construct an operadic morphism $\varphi_A:\mathcal{FP}\rightarrow End_A$, where $A$ is the set of strict order polynomials indexed by posets.
Consider a poset $P$ with $|P|$ points and associated order polynomials ${\somega[x]{P_1}, \somega[x]{P_2}, \cdots, \somega[x]{P_{|P|}}}$.
The operad $\mathcal{FP}$ assign to these order polynomials the element $\somega[x]{P(P_1, \cdots, P_{|P|})}$, which corresponds to the lexicographic sum of the posets $P_i$ induced by $P$. This assignment satisfies all the necessary conditions of an operadic morphism, establishing the desired result.
\end{proof}

The concept of Lexicographic sum is a well-established topic in order theory, with numerous studies on its properties and applications. For instance, Hiraguchi's relation provides a useful formula for calculating the dimension of the resulting poset:
 \[dim(P(P_1, \cdots, P_{|P|}))=\max\{dim(P),dim(P_1)\cdots,dim(P_{|P|})\},\]
where $dim$ denotes the dimension of a poset. However, instead of focusing on the specific properties of individual posets, our approach aims to identify common properties and structures shared among different poset algebras. 
 
 \subsection{Order series}\label{sectionOS}
There are two families of polynomials indexed by posets, strict order polynomials  $\somega[x]{P}=\sum d_{P,i}{x\choose i},$  and weak order polynomials
$\nsomega[x]{P}=\sum (-1)^{|P|-i}d_{P,i}\multiset{x}{i}$. 

The generating series associated with each of these order polynomials can be defined. The strict order series is given by $\zetaf[P]=\sum_{n=1}\somega[n]{P} x^n$, and the weak order series is defined as $\zetafp[P]=\sum_{n=1}\nsomega[n]{P} x^n$. To make the notation more concise, we will use $\zetaf[n]$ and $\zetafp[n]$ to denote $\zetaf[{\chain[n]}]$ and $\zetafp[{\chain[n]}]$, respectively.

In practice, the use of order series simplifies the computations and clarifies identities. In fact, the study of order series has led to the discovery of new binomial identities and generalizations of the Chu-Vandermonde identity, as shown in \cite{posets}. Therefore, in this section, we focus on exploring the properties of order series.

\begin{remark}\label{Rmk:basis}
Equation~\eqref{eqn:basic} reveals that the generating series of $\somega[x]{P}=\sum_{j=1}^{|P|} d_{P,j}{x\choose j}$ can be expressed as
\begin{eqnarray}
\zetaf[P]&=&\sum_{n=1}^\infty \somega[i]{P}x^i\nonumber\\
&=&\sum_{n=1}^\infty \sum_{j=1}^{|P|}d_{P,j}{i\choose j}x^i\nonumber\\
&=&\sum_{j=1}^{|P|}d_{P,j}\sum_{n=1}^\infty {i\choose j}x^i\nonumber\\
&=&\sum_{j=1}^{|P|}d_{P,j}\zetaf[i],\label{basis:1}
\end{eqnarray}
and similarly, for $\nsomega[x]{P}=\sum (-1)^{|P|-i}d_{P,i}\multiset{x}{i}$, we have $\zetafp[P]=\sum (-1)^{|P|-i}d_{P,i}\zetafp[i]$. We refer to the basis $\{\zetaf[n]\}_{n\in\mathbb{N}}$ (or equivalently $\{\zetafp[n]\}_{n\in\mathbb{N}}$) as the inclusion-exclusion basis for strict (resp. weak) order series. 
It is worth noting that while the inclusion-exclusion vectors $d_{i,P}$ were known to Stanley, to the best of our knowledge, the only reference where they are used to study order series/ Ehrhart series is the paper~\cite{posets}. This could be attributed to the fact that order polytopes possess favorable properties that we cannot expect in arbitrary integer polytopes.
\end{remark}

There is always a non-strict morphism from a poset to $\chain[1]$ that labels every element of the poset with the constant value $1$. Then, the series $\zetafp[P]$ starts at index $n=1$. For the strict order series $\zetaf[P]$, the lowest non-zero value is the length of a maximal chain on the poset.

In category theory, the Yoneda Lemma states that an object $c$ can be fully characterized by the functor $\hom(c,\_)$, which captures how the object interacts with its environment. In combinatorics, we similarly represent a poset $P$ by its order series $\zetaf[P]=\sum \#Hom(P,\chain[n]) x^n$. However, this series loses some information about the poset: it only considers morphisms from $P$ to chains and only keeps track of the number of such maps. Consequently, two posets can have the same order series, and such pairs are called Doppelgängers~\cite{Doppelgangers2, Doppelgangers}. For example, we have $\zetaf[\{w<x,w<y,w<z\}]=\zetaf[\{x<y, w<z\}]$.
In Lemma~\ref{Lemma:actionseries}, we show that the operad $\mathcal{FP}$, which acts on posets, also acts on order series,  bringing all the combinatorial structure to order series despite the information lost.

An operadic morphism $\varrho$ between $A$ and $B$ algebras over $\mathcal{FP}$, is a map making the following diagram commutative:
\[ \begin{tikzcd}
\mathcal{FP} \arrow{rd}{\varphi_B} \arrow[swap]{d}{\varphi_A} & \\%
End_{A}\arrow[swap]{r}{\varrho}&  End_{B}
\end{tikzcd}
\]

\begin{lemma}\label{Lemma:actionseries}
The operad $\mathcal{FP}$ acts on strict (weak) order series. The assignments
\begin{eqnarray}
  u:{x\choose n}&\mapsto&  \zetaf[n]\nonumber\\
    &=&\sum_{r=n}^\infty {r\choose n} x^r\nonumber\\
    &=&\frac{x^n}{(1-x)^{n+1}}\label{eqn:basic},\\ v:\multiset{x}{n}&\mapsto&\zetafp[n]\nonumber\\
    &=& \sum_{r=1}^\infty \multiset{r}{n}x^r\nonumber\\
    &=&\frac{x}{(1-x)^{n+1}},
    \label{eqn:zetaplus}
\end{eqnarray}
extend to an operadic isomorphism between algebras over the operad $\mathcal{FP}$.
\end{lemma}
\begin{proof}
We first extend the map $u$ linearly to obtain
\begin{equation*}
u\Big(\sum_{i=1}^{|P|} d_{P,i}{x\choose i}\Big)=\sum_{i=1}^{|P|} d_{P,i}\zetaf[i].
\end{equation*}
Given a poset $P$ and $P_1,\ldots,P_{|P|}$ finite posets, we define the action of $\mathcal{FP}$ on order series by
\begin{align*}
\varphi_s(P,\zetaf[P_1],\ldots,\zetaf[P_{|P|}])&=\sum_{i=1}^\infty \varphi_p(P,\somega[y]{P_1},\ldots,\somega[y]{P_{|P|}}|_{y=i})x^i.
\end{align*}
If we replace $P$ by $P(P_1,\ldots,P{|P|})$ in Equation~\eqref{basis:1}, we obtain the explicit description:
\begin{equation*}
\varphi_s(P,\zetaf[P_1],\ldots,\zetaf[P_{|P|}])=\sum_{j=1}^{|P(P_1,\ldots,P_{|P|})|}d_{P(P_1,\ldots,P_{|P|}),j}\zetaf[i].
\end{equation*}
Because the set $\{{x\choose i}\}_{i\in\mathbb{N}}$ is linearly independent over $\mathbb{Q}$, the map $u$ is a well-defined linear transformation. Moreover, because the set $\{\frac{x^i}{(1-x)^{i+1}}\}_{i\in\mathbb{N}}$ is linearly independent over $\mathbb{Q}$, if we have $\sum a_i u({x\choose i})=0$, then $a_i=0$, i.e., $u$ is injective and an isomorphism over its image. It follows that the action of the operad $\mathcal{FP}$ on strict (weak) order series is well-defined and isomorphic to the image of another algebra over $\mathcal{FP}$. The same argument applies to weak order series.
\end{proof}

We can leverage analytic methods in our study of posets via power series. For instance, given a poset $P$, we define the function $i(x)=\frac{1}{x}$ and the map $\iota(\zetafp[P])=(-1)^{|P|+1}(\zetafp[P]\circ{i})$. Note that we need to know that the function $\zetafp[P]$ is generated by the poset $P$ in order to evaluate $\iota$. 
\begin{lemma}\label{lemma:iso}
The map $\iota$ is an involution between the set of weak order series and the set of strict order series.
\end{lemma}
\begin{proof}
Let $\zetaf[P]$ be a strict order series. Then, we have
\begin{eqnarray*}
\iota(\zetaf[P])&=&(-1)^{|P|+1}(\zetaf[P]\circ{i})\\
&=&(-1)^{|P|+1}(\sum d_{P,i}\frac{x^i}{(1-x)^{i+1}})\circ{i}\\
&=&(-1)^{|P|+1}(\sum (-1)^{i+1}d_{P,i}\frac{x}{(1-x)^{i+1}})\\
&=&\sum (-1)^{|P|-i}d_{P,i}\frac{x}{(1-x)^{i+1}}\\
&=&\zetafp[P],
\end{eqnarray*}
where the second equality follows from substituting $x\rightarrow \frac{1}{x}$ and using the fact that $d_{P,i}=0$ for $i>|P|$. Thus, $\iota$ maps strict order series to weak order series.

Since $\iota$ is its own inverse, it is an involution between the set of weak order series and the set of strict order series.
\end{proof}

\begin{lemma}\label{lemma:i}
The morphism $\iota$ is an operadic involution between weak order series  and strict order series. 
\end{lemma}
\begin{proof} 
Let $P_1,\cdots, P_{|P|}$ be finite posets.  Now, to see that $\iota$ commutes with the operadic action, evaluate on $P(P_1,\cdots,P_{|P|})$ to obtain:
\begin{eqnarray*}
\iota(\varphi(P,\zetaf[P_1],\cdots,\zetaf[P_{|P|}]))&=&\iota(\zetaf[P(P_1,\cdots,P_{|P|})])\\
&=&\zetafp[P(P_1,\cdots,P_{|P|})]\\
&=&\varphi(P,\zetafp[P_1],\cdots,\zetafp[P_{|P|}])\\
&=&\varphi(P,\iota(\zetaf[P_1]),\cdots,\iota(\zetaf[P_{|P|}])).
\end{eqnarray*}
The verification of $\iota(\varphi(P,\zetafp[P_1],\cdots,\zetafp[P_{|P|}]))=\varphi(P,\iota(\zetafp[P_1]),\cdots,\iota(\zetafp[P_{|P|}]))$ fo\-llows the same reasoning.
\end{proof}

\begin{remark}\label{rmk:zero}
The series $\zetaf[\varnothing]=\frac{1}{(1-x)}$ is the unit of strict order series under the action of the operad $\mathcal{FP}$. The image of this unit under the involution $\iota$, denoted by $\zetafp[\mayadigit{0}] := (-1)^{1}(\zetafp[\varnothing]\circ\iota)=\frac{x}{1-x}$, it is the unit of weak order series under the action of the operad $\mathcal{FP}$. We introduce the symbol $\mayadigit{0}$ to represent an expansion around infinity of the power series $\zetaf[\varnothing]$.
\end{remark}

\subsubsection{Explicit Description of Low-Level Operations}\label{explicit}
We now describe the action of a poset $P$ on order series for small $P$. We denote the Hadamard product of two power series $f(x)$ and $g(x)$ as $f(x)\sqcup g(x)$. The table in Figure~\ref{fig:1} provides an explicit description of the action for series-parallel posets. For more information, please refer to \cite{posets}.

\begin{figure}[htb]
\begin{center}
\begin{tabular}{ | m{4cm} |  m{9cm} | } 
  \hline
  \hline
  Poset&Evaluation on order series\\
  \hline
  \hline
$\{x,y\}$&  $\{x,y\}(\zetaf[P_1],\zetaf[P_2])=\zetaf[P_1]\sqcup\zetaf[P_2]$\\ 
  \hline
  $\{x<y\}$ & $\{x<y\}(\zetaf[P_1],\zetaf[P_2])$\\
  &= $\zetaf[P_1](1-x)\zetaf[P_2]$ \\ 
  \hline
   $\{x_1<x_2, x_1<x_3\}$ & $\{x_1<x_2, x_1<x_3\}$ $(\zetaf[P_1],\zetaf[P_1],\zetaf[P_1])$\\
   &$=\zetaf[P_1](1-x)(\zetaf[P_2]\sqcup\zetaf[P_3])$\\ 
  \hline
\end{tabular}
\end{center}\caption{Operations on power series parametrized by posets. \label{fig:1}}
\end{figure}

Note that
\begin{eqnarray*}
\{x<y\}(\zetaf[k],\zetaf[j])&=&\zetaf[k](1-x)\zetaf[j]\\
&=&\zetaf[{k+j}],
\end{eqnarray*}
as expected. This is because the ordinal sum of posets satisfies $\{x<y\}(\chain[k],\chain[j])=\chain[k+j]$.

 Consider the following identity~\cite[(6.44)]{lcomb},\cite{posets}:
\begin{equation}\label{Eqn:1}
{x\choose n}{x\choose s}=\sum_{j=0}^{s} {n+j\choose s}{s\choose j}{x\choose n+j},
\end{equation}
where $s\leq n$.

According to~\cite{posets}, the Hadamard product of the order series of chains can be computed using the following formula:
\begin{equation}
\zetaf[n]\sqcup\zetaf[s] = \sum_{j=0}^{s}{n+j\choose s}{s\choose j}\zetaf[n+j],
\label{Eqn:cup}
\end{equation}
where $s\leq n$. This result can be proven using the fact that $Poly(P\sqcup Q)$ is the Minkowski sum of $Poly(P)$ and $Poly(Q)$, and then we count the simplices on the cannonical triangulation of $Poly(P\sqcup Q)$.

\begin{example}
\label{example}We define the action of the operad of finite posets on natural numbers by
considering the action of the operad on order series, where we record only the smallest chain after a poset acts on order series. For example, using Equation~\eqref{Eqn:cup}, we define $\{x,y\}(m,n)=\max\{m,n\}$, and using \[\{1<2\}(\zetaf[m],\zetaf[n])=  \zetaf[m+n],\] we define $\{1<2\}(m,n)=m+n$. By extending these definitions to all real numbers, we recover the tropical numbers~\cite{tropical} $(\mathbb{R}, +, \max, -\infty)$.

One could wonder about the actions of higher poset operations. For example, we compute $\{x<y>z<w\}(m,n,r,s)=\max(m+n,n+r,r+s)$. Strictly speaking, if we restrict to the operad generated by $\{x,y\}$ and $\{x<y\}$, then this is not a $4$-ary operation, but the result of evaluating a $6$-ary operation on duplicated inputs. This structure could be explained better with a product and permutation category (PROP), but we will not study such an approach in this paper.

In general, given a poset $P$, evaluating on $(n_1,\cdots, n_{|P|})$ returns the length of the maximal chain in the lexicographic sum $P(\chain[n_1],\cdots,\chain[n_{|P|}])$. Then, every poset operation on the tropical natural numbers is obtained by evaluating an operation involving only $\{1<2\}$ and $\{1, 2\}$, perhaps on duplicated inputs.
\end{example}

\begin{lemma}
\label{Lemma:4}
The operad of finite posets acting on order series contains a $4$-ary operation that is associative and non-commutative, and cannot be generated by only using the binary operations $\{x<y\}$ and $\{x,y\}$.
\end{lemma}

\begin{proof}
We used Mathematica~\cite{Mathematica} to compute the values of Table~\ref{fig:2}, as explained in the notebook~\cite{note}. Table~\ref{fig:2} contains the 4th-ary operations generated by $\{x,y\}$ and $\{x<y\}$, evaluated on $(\zetaf[1],\zetaf[1],\zetaf[1],\zetaf[1])$.

\begin{figure}[htb]
\centering
\begin{tabular}{ | m{4cm} |  m{7cm} | } 
  \hline
  Poset&Vector\\
  \hline  \hline
  $\{x<y<z<w\}$& 
$\zetaf[4]$\\ 
  \hline
$\{x<y<z,w\}$& $
3\zetaf[3]+4\zetaf[4]$\\ 
  \hline
$\{x<y,z<w\},$ $\{x<y,x<z, x<w\},$ $\{y<x,z<x, w<x\}$& 
$\zetaf[2]+6\zetaf[3]+6\zetaf[4]$\\ 
  \hline
$\{x,y,z<w\}$& 
$4\zetaf[2]+15\zetaf[3]+12\zetaf[4]$\\   \hline
$\{x,y,z,w\}$& 
$\zetaf[1]+14\zetaf[2]+36\zetaf[3]+24\zetaf[4]$\\ 
  \hline
  $\{x<y>z,w\},$ $\{x,y>z<w\} $&$2\zetaf[2]+9\zetaf[3]+8\zetaf[4]$\\   \hline
  $\{x<y>z<w\}$&$ 
\zetaf[2]+5\zetaf[3]+5\zetaf[4]$\\ 
  \hline
\end{tabular}
\caption{Vectors of the evaluation of the 4th-ary operation on $(\zetaf[1],\zetaf[1],\zetaf[1],\zetaf[1])$}
\label{fig:2}
\end{figure}

While $\{x<y\}$ and $\{x,y\}$ are binary, commutative, associative operations, Table~\ref{fig:2} shows that the operation $\{x<y>z<w\}$ is a non-commutative $4^{th}$-ary operation that is not generated by $\{x<y\}$ and $\{x,y\}$. We know that this operation is non-commutative because
\begin{eqnarray}
\{x<y>z<w\}(\zetaf[2],\zetaf[1],\zetaf[1],\zetaf[1])&=&\{x<y>z<w\}(\zetaf[1],\zetaf[1],\zetaf[2],\zetaf[1])\nonumber\\
&=&9\zetaf[5]+11\zetaf[4]+3\zetaf[2],\\
\nonumber\\
\{x<y>z<w\}(\zetaf[1],\zetaf[2],\zetaf[1],\zetaf[1])&=&\{x<y>z<w\}(\zetaf[1],\zetaf[1],\zetaf[1],\zetaf[2])\nonumber\\
&=&7\zetaf[5]+8\zetaf[4]+2\zetaf[2].\label{ope:4}
\end{eqnarray}

From Example~\ref{example}, one can wonder if there is an $n$-ary operation $f$ evaluated on duplicated inputs that is generated by $\{x,y\},\{x<y\}$ and sends $(\zetaf[1],\zetaf[1],\zetaf[1],\zetaf[1],\zetaf[1],\cdots)$ to $\zetaf[2]+5\zetaf[3]+5\zetaf[4]$. If such an operation exists, there is a series parallel poset $P_f$ associated with it. A series parallel poset is a poset generated by $\{x,y\},\{x<y\}$.

From \cite[Proposition 3.3]{posets}, parts 2) and 3) it follows that the maximal chain of $P_f$ has length 2 and the poset $P_f$ has 4 points. This rules out the possibility of having an operation with a higher number of inputs returning $\zetaf[2]+5\zetaf[3]+5\zetaf[4]$.
\end{proof}

Note that $\{x<y,x<z\}$ is also non-commutative since the value of $x$ cannot be exchanged with the value of $y$ or $z$. However, $\{x<y,x<z\}$ can be expressed as the composition (on $w$) of the binary operations $\{y,z\}$ and $\{x<w\}$.

Some know properties of $\{x<y>z<w\}$ are:
\begin{eqnarray*}
    \{x<y>z<w\}(\zetaf[a],\zetaf[b],\zetaf[c],\zetaf[d])&=&\{x<y>z<w\}(\zetaf[d],\zetaf[c],\zetaf[b],\zetaf[a])\\
    \\
      \{x<y>z<w\}(\zetaf[k],\zetaf[1],\zetaf[1],\zetaf[1])&=&(\frac{(k+3-2)(k+3+1)}{2}) \zetaf[k+3]\\
      &&+({k+2\choose 2}+\frac{(k+2-2)(k+2+1)}{2}) \zetaf[k+2]
      \\
      &&+{k+1\choose 2}  \zetaf[k+1]\\
      \\
      \{x<y>z<w\}(\zetaf[1],\zetaf[k],\zetaf[1],\zetaf[1])&=&(2(k+3)-3)\zetaf[k+3]+(2(k+2)-3)\zetaf[k+2]\\
      &&+[(k+2-1)]\zetaf[k+2]+(k+1-1)\zetaf[k+1]\\
      \\
      \{x<y>z<w\}(\zetaf[\varnothing],\zetaf[b],\zetaf[c],\zetaf[d])&=&\zetaf[c](1-x)(\zetaf[b]\sqcup \zetaf[d])\\
      \\
      \{x<y>z<w\}(\zetaf[a],\zetaf[\varnothing],\zetaf[c],\zetaf[d])&=&\zetaf[a]\sqcup \zetaf[c+d]\\
      \\
      \{x<y>z<w\}(\zetaf[a],\zetaf[b],\zetaf[\varnothing],\zetaf[d])&=&\zetaf[a+b]\sqcup \zetaf[d]\\
      \\
      \{x<y>z<w\}(\zetaf[a],\zetaf[b],\zetaf[c],\zetaf[\varnothing])&=&\zetaf[b](1-x)(\zetaf[a]\sqcup \zetaf[c])
\end{eqnarray*}

\section{The union of points}\label{sub:union}

In this section we study the case in which the poset is the disjoint union of points. As a result, we obtain conceptual proofs of several results on Eulerian numbers. 

Eulerian polynomials have been extensively studied and have numerous applications \cite{surveyeulerian, euler, eulerianB}. In this work, we adopt the modern definition of Eulerian polynomials, which is given by the following generating function:
\begin{equation}\label{eq:euler}
\sum_{n=0}^\infty A_n(t)\frac{x^n}{n!}=\frac{t-1}{t-e^{(t-1)x}}.
\end{equation}
By multiplying Equation~\eqref{eq:euler} by $x$ and substituting $x=\frac{1}{1-t}$, we obtain
\begin{equation}\label{Eq:Eu}
\sum_{n=0}^\infty \frac{A_n(t)}{(1-t)^{n+1}n!}=\frac{1}{e^{-1}-t}.
\end{equation}
It is well known that the Eulerian polynomials satisfy the identity:
\begin{equation}
x\frac{A_n(x)}{(1-x)^{n+1}}=
\sum_{k=1}^\infty k^n x^k=\sqcup_n \zetaf[{1}].\label{eq:polyid}
\end{equation}

 Here, we realize the right side of the equation as the strict order series of the disjoint union of $n$ points. 
 We consider the $n$-cube as the order polytope of the disjoint union of $n$ points. 

Given a polytope $Po$, we can define $L(Po,t)$ as the count of lattice points in the expansion of $tPo$ for a given value of $t$. By summing $L(Po,t)$ over all positive integers $t$, we arrive at the Ehrhart series of $Po$, which can be expressed as $\sum_{t=1}^\infty L(Po,t) x^t=\frac{h^\ast(x)}{(1-x)^{d+1}}$, where $h^\ast(x)$ is a polynomial of degree less than $d$. 
To begin our exploration, let us consider a poset $P$. By computing the Ehrhart series of the order polytope of $P$, we can then utilize Equation~\eqref{eqn:e} to deduce the strict order series of $P$:

\begin{equation}\label{eqn:e}
\zetafp[P]=x\frac{h^*(x)}{(1-x)^{d+1}}.
\end{equation}
When $P$ is the
disjoint union of $n$ points, then the $n$-cube serves as the order polytope of $P$. 

  \begin{proposition}
 For a poset $P$, the Ehrhart series of the order polytope of $P$ $\frac{h^*(x)}{(1-x)^{\|P\|+1}}$ satisfies $h^*_{\|P\|}=0$.
 \end{proposition}
\begin{proof}
Note that $\frac{h_{\|P\|}^*x^{\|P\|+1}}{(1-x)^{\|P\|+1}}$ cannot be in the algebra $\{\zetafp[i]=\frac{x}{(1-x)^{i+1}}\}_{1\leq i\leq \|P\|}$ unless $h_{\|P\|}^*=0$.
\end{proof}

By utilizing Equation~\eqref{eq:polyid}, we can see that the $h^*(x)$ polynomial of the $n$-cube is, in fact, the $n-$Eulerian polynomial $A_n(x)={\sum_{m=0}^{n-1}A(n,m)}x^m$.

Further applying $x=-1, t=-1$ and $\psi_n(t):=A_n(-t)$ to Equation~\eqref{eq:polyid}, we can deduce the following entry from \cite[Chapter 5, Entry 4]{R}:
For $|P|<1$ and $n\geq0$,
\[(p+1)^{-n-1}\psi_n(p)=\sum_{k=0}^\infty(k+1)^n(-p)^k.\]

\begin{example}
By combining Lemma~\ref{IEP} and Remark~\ref{Rmk:basis}, we can express the order series of the union of $n$ points as:
\begin{equation}\label{eq:union}
\sqcup_n\zetaf[{1}] = \sum_{i=1}^{n} d_{\sqcup_n\chain[1],i}\zetaf[i],
\end{equation}
where $d_{\sqcup_n\chain[1],i}$ is the number of internal $i$-faces in the canonical triangulation of the $n$-cube that are obtained as the intersection of $n-i+1$ maximal $n$-simplices.
\end{example}

\begin{corollary}\label{Cor:up}
For a fixed integer $n$, there exists a sequence of natural numbers ${d_{\sqcup_n\chain[1],i}}{1\leq i\leq n}$ such that for all $r>0$,
\begin{itemize}
\item $r^n=\sum_{i=1}^{n} d_{\sqcup_n\chain[1],i}{r\choose i}$,
\end{itemize}
where ${q\choose p}=0$ if $q<p$.
\end{corollary}
\begin{proof}
By using Equation~\eqref{eqn:basic} on Equation~\eqref{eq:union} and comparing the coefficients with Equation~\eqref{eq:polyid}, the result follows.
\end{proof}

This result is a significant improvement over Ramanujan's result, which is limited to the case where $1\leq r\leq n$ \cite[Chapter 5, Corollary 6, i)]{R}. The numbers $d_{\sqcup_n\chain[1],i}$ are closely related to the Stirling numbers of the second kind $S(n,i)$, which count the number of non-empty partitions of a set with $n$ labeled objects. Specifically, we have $d_{\sqcup_n\chain[1],i} = i!S(n,i)$ \cite[Theorem 5.6]{EulerandStirling}. The values of $d_{\sqcup_n\chain[1],k}$ for $1\leq k\leq n$ are explicitly known: for $n=2$, they are ${1, 2}$, and for $n=3$, they are ${1, 6, 6}$, which correspond to the sequence A019538 in the Online Encyclopedia of Integer Sequences \cite{A019538}.

In the following proposition we use our poset point of view to provide a conceptual proof the equality of two power series.
\begin{proposition}\label{Lemma:pos}
For any $n\in\mathbb{N}$, we have:
\begin{equation}\label{eqn:pos}
\sqcup_n \zetafp[1]=\sqcup_n \zetaf[1].
\end{equation}
\end{proposition}
\begin{proof}
A disjoint union of points cannot distinguish between strict and non-strict morphisms. Therefore, the weak and the strict order series are the same. In other words, for the union of points, every weak order-preserving map is also a strict order-preserving map. 
\end{proof}

\begin{corollary}\label{lemma:aux}
Let $P=\sqcup_n\chain[1]$, then
\[\sum_{k=1}^{n} k!S(n,k)x^{k-1}(1-x)^{n-k}=\sum_{k=1}^n (n-k)!S(n,n-k)(1-x)^{k}.\]
\end{corollary}

\begin{proof}
Starting from $\zetaf[P]=\zetafp[P]$ and using Lemma~\ref{IEP}, we obtain:
\[\sum_{k=1}^{n} \frac{k!S(n,k)x^k}{(1-x)^{k+1}}=\sum_{k=1}^{n} (-1)^k\frac{n-k!S(n,n-k)x}{(1-x)^{n-k+1}},\]
Multiplying both sides by $\frac{(1-x)^{n+1}}{x}$ yields the desired identity.
\end{proof}

Consider Ramanujan's notation $\psi_n(p)=\sum_{k=0}^{n-1}F_{k+1}(n)(-p)^k$.
We are now ready to prove the following statement from \cite[Chapter 5, Entry 6]{R}:
\begin{entry}
Let $1\leq r\leq n.$ Then
\begin{enumerate}
\item $F_r(n)=F_{n-r+1}(n)$,
\item \label{e:W}$r^n=\sum_{k=0}^{r-1} F_{r-k}(n) {n+k\choose k},$
\item \label{e:E} $F_r(n)=\sum_{k=0}^{r-1} (-1)^k{n+1\choose k}(r-k)^n.$
\end{enumerate}
\end{entry}

The identity~\eqref{e:W} is due to Worpitzky~\cite{w} while the identity~\eqref{e:E} is due to Euler~\cite{euler}.
Remember that the $h^*$ polynomial of the $n$--cube 
$  \sqcup_n \zetafp[1]=x\frac{h^*(x)}{(1-x)^{n+1}}$ is the Eulerian polynomial $\sum_{i=0}^{n-1} A(n,i)x^i$.

\begin{corollary}\label{Cor:Ehrhart} Let $1\leq r\leq n$.
Then
\begin{enumerate}
    \item\label{ref-i} $A(n,i)=A(n,n-1-i).$
    \item\label{ref-ii} $k^n=\sum_{i=0}^{n-1} A(n,i) {k+i\choose n}, k^n=\sum_{i=0}^{n-1} A(n,i)\multiset{k-i}{ n}$.
    \item\label{ref-iii} $A(n,i)=\sum_{r=0}^{n+1} (-1)^r{n+1\choose r}(i+1-r)^n.$
\end{enumerate}
\end{corollary}

\begin{proof}
$i)$ The right-hand side of Equation~\eqref{eqn:pos} is related to the left-hand side via the Lemma~\ref{lemma:iso}. We then replace 
Equation~\eqref{eqn:e} 
to obtain:
\begin{eqnarray*} x\frac{\sum_{i=0}^{n-1}A(n,i) x^{i}}{(1-x)^{n+1}}
&=&
(-1)^{n+1} \frac{1}{x}\frac{\sum_{i=0}^{n-1}A(n,i) \frac{1}{x^{i}}}{(1-\frac{1}{x})^{n+1}}\\
&=&
\frac{\sum_{i=0}^{n-1}A(n,i) {x}^{n-i}}{(1-{x})^{n+1}}.
\end{eqnarray*}
Then $\ref{ref-i})$ follows by comparing the $h^*$ vectors degree wise.
$ii)$ Worpitzky identity follows from comparing degree-wise the coefficients of 
$\sum_{k=1}^\infty k^n x^k=\frac{\sum_{i=0}^{n-1}A(n,i) x^{n-i}}{(1-x)^{n+1}}.$ The multiset version follows from $\sum_{k=1}^\infty k^n x^k=x\frac{\sum_{i=0}^{n-1}A(n,i) x^{i}}{(1-x)^{n+1}},$ but due to the symmetry of the coefficients, it is the same identity. $iii)$ Follows the same technique as in \cite{R}.
\end{proof}

\section{Poset version of Ramanujan's results}\label{sec:Poset}
In this section we show that certain results of Ramanujan are parametrized by disjoint union of posets. We say that we obtained a poset version of Ramanujan results when we prove these results for arbitrary finite posets.

Ramanujan gives the following examples:
\begin{entry}[{\cite[Chapter 5, Example 1,2.]{R}}]
\[\sum_{k=1}^{\infty}\frac{k^5}{2^k}=1082, \sum_{k=1}^{\infty}\frac{k^5}{3^k}=273/4.\]
\end{entry}

They can be justified by using $P=\sqcup^5 \chain[1]$ on the next lemma:
 \begin{lemma}\label{lemma:comp} Let $P$ be a finite poset and $\zetaf[P]=\sum_{i=1}^{\|P\|} d_{P,i}\zetaf[i]$.
Let $r>1$, and $r_0$ the length of a maximal chain in $P$, then:
\begin{eqnarray}
\sum_{n=r_0}^\infty \frac{\somega[n]{P}}{r^n}=\sum_{i=1}^{\|P\|} (-1)^{i+1}d_{P,i}\frac{r}{(1-r)^{i+1}}=(-1)^{\|P\|+1}\zetafp[P],\label{Eqn:main}\\
\sum_{n=1}^\infty \frac{\nsomega[n]{P}}{r^n}=(-1)^{\|P\|+1}\sum_{i=1}^{\|P\|} d_{P,i}\frac{r^i}{(1-r)^{i+1}}=(-1)^{\|P\|+1}\zetaf[P].\label{Eqn:second}
\end{eqnarray}
\end{lemma}

\begin{proof}
From Lemma~\ref{IEP} we obtain the equality 
\[(-1)^{\|P\|+1}\sum_{k=r_0} {\somega[k]{P}}{x^k}=(-1)^{\|P\|+1}\sum_{i=1}^{\|P\|} d_{P,i}\frac{x^i}{(1-x)^{i+1}}.\] 
Note that $\sum \frac{k^{\|P\|}}{x^k}$ converges on $|x|>1$, We now apply $\iota$ (from Lemma~\ref{lemma:iso}) on both sides to obtain
\[\sum_{k=r_0} \frac{\somega[k]{P}}{x^k}=\sum_{i=1}^{\|P\|} (-1)^{i+1}d_{P,i}\frac{x}{(1-x)^{i+1}}\]
This proves \eqref{Eqn:main}. The identity \eqref{Eqn:second} follows from the same reasoning.
\end{proof}

\begin{example}
For $P=\chain[1]*(\chain[1]\sqcup\chain[1]\sqcup\chain[1]),$ we calculated $ \zetafp[P]=\zetafp[2]-6\zetafp[3]+6\zetafp[4]$. 
Then using Equation~\eqref{eqn:zetaplus} we compute \[\somega[k]{P}={k\choose 2}+6{k\choose 3}+6{k\choose 4}=\frac{k^4-2k^3+k^2}{4}.\] Evaluating Equation~\eqref{Eqn:main} at $r=5$, we obtain 
\begin{eqnarray*}
\sum_{k=2}^\infty \frac{k^4-2k^3+k^2}{4(5^k)}&=&-\frac{5}{(1-5)^3}+6\frac{5}{(1-5)^4}-6\frac{5}{(1-5)^5}\\
&=&\frac{115}{512},
\end{eqnarray*}
similarly,  \[\nsomega[k]{P}={k+1\choose 2}-6{k+2\choose 3}+6{k+3\choose 4}=\frac{k^2(k+1)^2}{4},\]
\begin{eqnarray*}
\sum_{k=1}^\infty \frac{k^2(k+1)^2}{4*5^k}&=&-\frac{5^2}{(1-5)^3}+6\frac{5^3}{(1-5)^4}-6\frac{5^4}{(1-5)^5}\\
&=&\frac{575}{512}.
\end{eqnarray*}
\end{example}

Consider the Riemann zeta function $\zeta(i)=\sum_{n=1}^\infty \frac{1}{n^i}$, whose values on even natural numbers are well-known. The Bernoulli numbers are defined by the generating function $\frac{x}{1-e^x}=\sum_{m=0}^\infty B_m\frac{x^m}{m!}$, and Euler famously showed that $\zeta(2n)=-\frac{(2\pi i)^{2n}B_{2n}}{2(2n)!}$.

\begin{remark}\label{rmk:zs}
Although the irrationality of $\zeta(3)$ is established~\cite{z3}, the status of the odd zeta values remains elusive. It is known that one out of $\zeta(5), \zeta(7), \zeta(9), \zeta(11)$ is irrational~\cite{z57911}, and at least two out of $\zeta(5),\zeta(7),\cdots,\zeta(35)$ are irrational~\cite{atleastwo, oddzeta}. In fact, an elementary proof that one out of $\zeta(5),\zeta(7),\cdots,\zeta(25)$ is irrational is provided in~\cite{elementaryz}.
\end{remark}
 Despite the lack of complete understanding of odd zeta values, remarkable identities involving these numbers can still be established.

\begin{entry}[{\cite[Chapter 5, Collolary 6 iv)]{R}}]\label{Entry:R} Fix $n$ natural number. Then there are $\{A_k\}_{1\leq k\leq n}\in\mathbb{Z}$ such that
\[\sum_{k=1}^{\infty}(-1)^{k+1}k^n(\zeta(k+1)-1)=(-1)^{n}+(-1)^n2^{-n-1}\psi_n(1)+\sum_{k=1}^{n}(-1)^{k+1}A_k\zeta(k+1).\]
\end{entry}

Consider  $\sum_k (-1)^{k+1}p(k)(\zeta(k+1)-1)$ where $p(t)=\sum_{n=1}^N b_nt^n$ is a polynomial with rational coefficients. Using Ramanujan's results we obtain 
\begin{eqnarray}
\sum_k (-1)^{k+1}p(k)(\zeta(k+1)-1)&=&\sum_{n=1}^N b_n\sum_k (-1)^{k+1}k^n(\zeta(k+1)-1)
\nonumber\\
&=&\sum_{i=1}^{N} b^\prime_i\zeta(i+1)+b_0^\prime.\label{eqn:line}
\end{eqnarray}
  However, we do not know much about the corresponding coefficients $b_i^\prime$.  In the following theorem we generalized Ramanujan identity by interpreting $k^n$ as the orden polynomial of disjoint union of $n$-points. Then, when the polynomial is the order polynomial of a finite poset, we show that the coefficients of the sum in the base $\zeta(2),\cdots,\zeta(n+1)$ are integers multiplied by alternating signs.

 \begin{theorem}\label{theorem:R2}
Fix a finite poset $P$, let $r_0$ the length of a maximal chain in $P$, and let $n=\|P\|$. There are $\{d_{P,i}\in\mathbb{N}\}_{1\leq i\leq n}$ such that

\begin{equation}\label{eqn:womega}
\sum_{k=r_0}^\infty {(-1)^{k+1}\somega[k]{P}}(\zeta(k+1)-1)=\sum_{i=1}^{\|P\|} (-1)^{i+1}d_{P,i}(\zeta(i+1)-1-\frac{1}{2^{i+1}}).
\end{equation}
\end{theorem}

\begin{proof}
We divide by $r$ on both sides of Equation~\eqref{Eqn:main}, obtaining the Ehrhart series of the order polytope of $P$ in the inclusion exclusion basis, we evaluate on $-r$ to obtain:

\begin{eqnarray}\label{eqn:ffe}
  \sum_{k=1}^\infty (-1)^{k+1}\somega[k]{P}\frac{1}{r^{k+1}}=\sum_{i=1}^{|P|} (-1)^{i+1}d_{P,i}\frac{1}{(1+r)^{i+1}}.  
\end{eqnarray}

\begin{eqnarray*}
\sum_{k=r_0}^\infty {(-1)^{k+1}\somega[k]{P}}(\zeta[k+1]-1)&=&
\sum_{r=2}^{\infty}\sum_{k=r_0}^\infty \frac{\somega[k]{P}}{(-r)^{k+1}}\\
&=&\sum_{r=2}^{\infty} \sum_{i=1}^{\|P\|} \frac{(-1)^{i+1}d_{P,i}}{(1+r)^{i+1}}\\
&=&\sum_{i=1}^{\|P\|} (-1)^{i+1}d_{P,i}(\zeta[i+1]-1-\frac{1}{2^{i+1}}).
\end{eqnarray*}
\end{proof}

Consider an identity that involves only expressions of the form $\frac{1}{x^\alpha},\frac{1}{(x \pm 1)^\beta},\cdots$, where the exponents are natural numbers. If the identity holds for all values of $x$ greater than or equal to a fixed $n$, and the sum of these expressions converges, then by summing these identities, one can obtain an identity involving zeta values. This process, referred to as expressing a relation in finite form, has been discussed in \cite{Series}. It is worth noting that this technique can be used to evaluate every possible subset of the real numbers, as long as the limit of partial sums converges.

Since zeta values hold great significance, we are particularly interested in this type of identity. The question that arises is how to discover more such identities involving fractions $\frac{1}{x^\alpha},\frac{1}{(x \pm 1)^\beta},\cdots$. One approach is to draw from the theory of posets, as we have done to obtain families of such relations, the Equation~\eqref{eqn:womega} is the finite form expression  of Equation~\eqref{eqn:ffe}. 

\begin{example}
For $P=\chain[1]*(\chain[1]\sqcup\chain[1]\sqcup\chain[1])$, we obtain $d_{P,1}=0, d_{P,2}=1, d_{P,3}=6,$ and $ d_{P,4}=6$:
\begin{eqnarray*}
\sum_{n=2}^\infty(-1)^{k+1} \frac{k^4-2k^3+k^2}{4}(\zeta(k+1)-1)&=&-(\zeta(3)-1-\frac{1}{2^3})+6(\zeta(4)-1-\frac{1}{2^4})\\
&&-6(\zeta(5)-1-\frac{1}{2^5}).\end{eqnarray*}
\end{example}

\section{Action of the operad of posets on zeta values}\label{Sec:spz}
Christian Goldbach (1690-1764) showed that \[\sum_{n=2}^\infty (\zeta(n)-1)=1.\]
In \cite{compuzeta} series with rational coefficients and terms  $\zeta(n)-1$ are called rational zeta series, see also~\cite{otherref, otherseries, spiral,divseries, expansions, sumsS}.
 
Consider  
\begin{eqnarray}\sum_{n=k}^\infty {n\choose k}(\zeta(n+1)-1)=\zeta(k+1),\label{eqn:1}\end{eqnarray} discovered by \cite{compuzeta}.

For sums $\sum_{i=0} a_i {x\choose i}$ with $a_i\in\mathbb{Q}$ and $a_i=0$ except for a finite number of terms, we define
\begin{eqnarray*}
\tilde{n}(\sum_{i=0} a_i {x\choose i})&=&\sum_{n=1}^\infty (\sum_{i=0} a_i {n\choose i})(\zeta(n+1)-1)\\
&=&\sum_{i=0} a_i \zeta(i+1).
\end{eqnarray*}

From Theorem~\ref{theorem:R2} we deduce the following identity:
\begin{eqnarray}
 \sum_{n=k}^\infty (-1)^{n+1} {n\choose k}(\zeta(n+1)-1)&=&(-1)^{k+1}(\zeta(k+1)-1-1/2^{k+1}), \label{Eq:1}\end{eqnarray}
together with
\[\sum_{n=1}^\infty (-1)^{n+1}(\zeta(n+1)-1)=1/2,\]
they define a linear function $\tilde{n}_2$  that sends a finite sum with rational coefficients $\sum_{i=0} a_i {x\choose i}$ with $a_i=0$ except for a finite number of terms, into
\begin{eqnarray*}
    \sum_{k=1}^\infty (-1)^{k+1} (\sum_{i=0} a_i {k\choose i})(\zeta(k+1)-1)&=&\frac{a_0}{2}+\sum_{i=1}^r (-1)^{i+1}a_i(\zeta(i+1)-1-\frac{1}{2^{i+1}}).
\end{eqnarray*}

We will use the functions $\tilde{n}, \tilde{n}_2$ to define the action of the operad of posets on zeta values.

From Theorem~\ref{thm:main} we know there is an action of the operad of posets on the vector space $\mathbb{Z}[\{{x\choose i}\}_{i\in\mathbb{N}}]$, we show how to transfer the action to $\mathbb{Z}[\{\tilde{n}({x\choose i})\}_{i\in\mathbb{N}}]$ and to $\mathbb{Z}[\{\tilde{n}_2({x\choose i})\}_{i\in\mathbb{N}}]$. The linear map $\tilde{n}$ will become trivially an operadic morphism. 

For a poset $P$, with order series  $\somega[x]{P}=\sum_{i=1}^{|P|} d_{P,i}{x\choose i}$, we define $\zetafpn[P]:=\tilde{n}(\somega[x]{P})$ and $\zetafn[P]:=(-1)^{|P|+1}\tilde{n}_2(\somega[x]{P})$.

\begin{corollary}\label{CorZeta}
The operad $\mathcal{FP}$ acts non-trivially on the numbers $\{\zetafn[P]\}_{P \hbox{ finite poset}}$ and on $\{\zetafpn[P]\}_{P \hbox{ finite poset}}$.
\end{corollary}

\begin{proof}
The action of a poset $P$ on $\zetafn[P_1],\cdots,\zetafn[P_{|P|}]$,  is given by 
\[(-1)^{|P(P_1,\cdots,P_{|P|})|+1}\tilde{n}_2(\somega[x]{P(P_1,\cdots,P_{|P|})}).\]

We see that the action is not trivial since  $\somega[x]{\chain[k]}$ is assigned to $(\zeta(k+1)-1-\frac{1}{2^{k+1}})$, and different chains are sent to different zeta values.
We see that the action of the operad of posets is different from the action of the operad of series parallel posets because from Lemma~\ref{Lemma:4}, the action of $\{x<y>z<w\}$ is not equal to the action of any quaternary operation generated by $\{x,y\}$ and $\{x<y\}$. 

The same arguments apply to $\zetafpn[P]$.
\end{proof}

We describe the action of a poset in $\mathbb{Z}\{{x\choose i}\}$ by computing the integer coordinates, we call these coordinates structural constants. We use this point of view to find the action on zeta values of the operad of finite posets. First compute the constants on polytopes or order series, and then the map $\tilde{n}, \tilde{n}_2$ force the constants to appear on the zeta value side. For example, the order polytope of the Minkowski sum of $\Delta[2]$ and $\Delta[m]$ can be written as $\cup_{m+1}\Delta[m+1]\setminus (\cup_m \Delta[m])$. The action of $\{x,y\}$ on posets $P, Q$ is isomorphic to the action of the Minkowski sum on the corresponding order polytopes~\cite{posets} $Poly(P), Poly(Q)$. Then we use $\tilde{n}_2 $ to conclude that 
 $\{x,y\}(\zetafn[2],\zetafn[m])
$ is equal to
\begin{equation*}
=-m(\zeta(m+1)-1-\frac{1}{2^{m+1}})
+(m+1)(\zeta(m+2)-1-\frac{1}{2^{m+2}}).
\label{eq:Ex2}
\end{equation*}

Using Equation~\eqref{ope:4} we can evaluate the associative non-commutative quaternary operation $\{x<y>z<w\}$ on $(\zetafn[1],\zetafn[2],\zetafn[1],\zetafn[1])$ and we obtain
\begin{eqnarray*}
\{x<y>z<w\}(\zetafn[1],\zetafn[2],\zetafn[1],\zetafn[1])&=&2\zetafn[2]-8\zetafn[3]
+5\zetafn[4].
\end{eqnarray*}

%
In addition to the structural constants, the order series inherits calculus identities that are valid in this context. For example, consider the following proposition from \cite[Proposition 4.1]{posets}:

\begin{eqnarray*}
\zetaf[s]\sqcup (\zetaf[p]*\zetaf[q])&=&\sum_{a+c=s}(\zetaf[a]\sqcup\zetaf[p])*(\zetaf[c]\sqcup\zetaf[q])
\nonumber\\
&-&[\sum_{a+c=s-1}(\zetaf[a]\sqcup\zetaf[p])*(\zetaf[c]\sqcup\zetaf[q])]*\zetaf[1],
\end{eqnarray*}
This proposition can be proven by understanding $\sqcup\zetaf[n]$ as a differential operator. We work with order series instead of order polynomials because we frequently encounter similar expressions that are better understood with the power series structure.

At the level of zeta values, we obtain:
\begin{eqnarray}
\zetafn[s]\sqcup (\zetafn[p]*\zetafn[q])&=&\sum_{a+c=s}(\zetafn[a]\sqcup\zetafn[p])*(\zetafn[c]\sqcup\zetafn[q])
\nonumber
\end{eqnarray}
\begin{eqnarray}
&-&[\sum_{a+c=s-1}(\zetafn[a]\sqcup\zetafn[p])*(\zetafn[c]\sqcup\zetafn[q])]*\zetafn[1],
\end{eqnarray}
similarly for $\zetafpn[P]$.

These identities are equivalent to binomial identities. 



It is natural to wonder whether an arbitrary sum of the form $b_1\zetafn[2]+\cdots+ b_{n-1}\zetafn[n]$, where $b_i$ are integer coefficients, can be expressed as $\zetafn[P]$ for some poset $P$. An algorithm for finding a series-parallel poset was proposed in~\cite[Section 3.1]{posets}, with optimizations for Wix\'arika posets.

Finally, we relate the injectivity of $\tilde{n}$ and $\tilde{n}_2$ to a well-known open conjecture in number theory~\cite{colombianotes}:

\begin{theorem}\label{thm:li}
The following statements are equivalent:
\begin{enumerate}
\item \label{Conjec:1}
The numbers $\{1,\zeta[2],\zeta[3],\cdots\}$ are linearly independent over $\mathbb{Q}$.
\item The functions $\tilde{n}$ and $\tilde{n}_2$ are injective. \label{Conject:2}
\end{enumerate}
\end{theorem}

\begin{proof}
Assuming~\ref{Conjec:1}.  Since the order polynomials are integer valued, from linear independence we deduce that if there are two polynomials in which $\tilde{n} (\tilde{n}_2)$ return the same value, each coordinate $\zeta(i+1)-1$ must have the same coefficient.
Now, for different polinomials, there is a coordinate in which the values of $\tilde{n}  (\tilde{n}_2)$ differ, otherwise the order polynomials agree on an infinite number of integers, but this implies that the polynomials are the same. Then \ref{Conject:2} holds.

Assuming~\ref{Conject:2}. Given a linear combination with rational coefficients, $\sum a_{i}{\zeta[i]}=0$. We can multiply by an integer to obtain integer coefficients: $\sum b_{i}{\zeta[i]}=0$. The preimage of this element $\sum b_i{x\choose i}$ belongs to the kernel of the map, which is zero. But from $\sum b_i{x\choose i}=0$ we deduce that $b_i=0$, due to linear independence of $\{{x\choose i}\}$.
\end{proof}


Regarding the first part of the previous theorem, Rivoal~\cite{infzeta} and Ball-Rivoal~\cite{indep} established that an infinite number of odd zeta values are linearly independent. Several recent works have made significant improvements on the Ball-Rivoal Theorem, including~\cite{infzetaoddnew, state, trascendental, oddzeta, elementary}. See also Remark~\ref{rmk:zs}.
\section{Further remarks}\label{Sec:pa}
\begin{itemize}
    \item Consider the following theorem:
\begin{theorem}
If $n$ is a natural number, then
\[\sum_{n=1}^{\infty} \frac{1}{n^k(n+1)^k}=\sum^k_{ n=0,\\ n\neq k-1 }(1+(-1)^{k-n})\zeta({k-n}){-k\choose n}.\]
\end{theorem}
This result can be found in  \cite[Chapter 7, Entry 22]{R},  
and \cite{Series}. Note that \[\sum_n \frac{1}{n^k(n+1)^k}=\sum_{n<n_2}\frac{1}{n^kn_2^k}-\sum_{n<n_2}\frac{1}{n^k(n_2+1)^k}.\] If we write the values \[\sum_{n_1<\cdots<n_r}\frac{1}{(n_1+c_1)^{k_1}\cdots (n_r+c_r)^{k_r}}, c_1\cdots,c_r\in\mathbb{N},\] as a sum of zeta values and a rational number, then we believe the coefficients should have a poset interpretation.

\item Is there a generating function for the coefficients $d_{P,i}$ where $P$ runs over all finite posets with $n$ points?
 
\end{itemize}

\section*{Acknowledgements}
We would like to express our gratitude to Antonio Arciniega-Nevarez and Marko Berghoff for their valuable feedback on earlier versions of this document. We also extend our thanks to Greta Panova and the anonymous reviewers for their insightful comments. We are especially grateful to Hansol Hong for his unwavering support and encouragement throughout the project. Additionally, we would like to acknowledge the support provided by the faculty members of Northwestern University and the CINVESTAV over the years. Their contributions have been invaluable to the development of this work. We used the software \cite{MATLAB,Mathematica} for several experiments.

\section{Funding}
The first author's research was made possible through the generous support of the National Research Foundation of Korea (NRF) grant funded by the Korean government (MSIT) (No. 2020R1C1C1A01008261). JLMC's research was supported by start-up funds from Michigan State University (MSU).

\bibliographystyle{abbrv}
\bibliography{references}

\begin{thebibliography}{10}

\bibitem{otherref}
V.~S. Adamchik and H.~M. Srivastava.
\newblock Some series of zeta and related functions.
\newblock {\em Analysis}, 18:131–144, 1998.

\bibitem{z3}
R.~Ap{\'e}ry.
\newblock Irrationalit\'e de $\zeta 2$ et $\zeta 3$.
\newblock In {\em Journ\'ees Arithm\'etiques de Luminy}, number~61 in
  Ast\'erisque. Soci\'et\'e math\'ematique de France, 1979.

\bibitem{posets}
J.~A. Arciniega-Nevarez, M.~Berghoff, and E.~Dolores-Cuenca.
\newblock An algebra over the operad of posets and structural binomial
  identities.
\newblock {\em Bol. Soc. Mat. Mex.}, 29, 2023.

\bibitem{indep}
K.~Ball and T.~Rivoal.
\newblock Irrationality of infinitely many values of the zeta function at odd
  integers.
\newblock {\em Invent. Math.}, 146(1):193--207, 2001.

\bibitem{crt}
M.~Beck and R.~Sanyal.
\newblock {\em Combinatorial Reciprocity Theorems: An Invitation to Enumerative
  Geometric Combinatorics}, volume 195.
\newblock American Mathematical Society, Providence, Rhode Island, first
  edition, 2018.

\bibitem{R}
B.~C. Berndt.
\newblock {\em Ramanujan’s Notebooks, Part I}.
\newblock Springer, New York, NY, 1 edition, 1985.

\bibitem{compuzeta}
J.~M. Borwein, D.~M. Bradley, and R.~E. Crandall.
\newblock Computational strategies for the riemann zeta function.
\newblock {\em Journal of Computational and Applied Mathematics},
  121(1):247--296, 2000.

\bibitem{otherseries}
K.~Boyadzhiev.
\newblock A special constant and series with zeta values and harmonic numbers.
\newblock {\em Gazeta Matematica, Seria A}, 115:1--16, 2018.

\bibitem{fvect}
F.~Breuer.
\newblock Ehrhart {{\(f^*\)}}-coefficients of polytopal complexes are
  non-negative integers.
\newblock {\em Electron. J. Comb.}, 19(4):research paper p16, 22, 2012.

\bibitem{spiral}
D.~Brink.
\newblock The spiral of theodorus and sums of zeta-values at the half-integers.
\newblock {\em American Mathematical Monthly}, 119, 11 2012.

\bibitem{colombianotes}
F.~Brown.
\newblock Iterated integrals in quantum field theory.
\newblock 2009.

\bibitem{Doppelgangers2}
T.~Browning, M.~Hopkins, and Z.~Kelley.
\newblock Doppelgangers: the \uppercase{U}r-operation and posets of bounded
  height, 2018.

\bibitem{IVP}
P.-J. Cahen and J.-L. Chabert.
\newblock {\em Integer-valued polynomials}, volume~48 of {\em Math. Surv.
  Monogr.}
\newblock Providence, RI: American Mathematical Society, 1997.

\bibitem{AMS}
P.-J. Cahen and J.-L. Chabert.
\newblock What you should know about integer-valued polynomials.
\newblock {\em Am. Math. Mon.}, 123(4):311--337, 2016.

\bibitem{divseries}
B.~Candelpergher.
\newblock {\em {Ramanujan summation of divergent series}}, volume 2185.
\newblock {Lectures notes in mathematics}, Mar. 2017.

\bibitem{surveyeulerian}
L.~Carlitz.
\newblock Eulerian numbers and polynomials.
\newblock {\em Mathematics Magazine}, 32(5):247--260, 1959.

\bibitem{carr}
G.~Carr.
\newblock {\em Formulas and theorems in pure mathematics}.
\newblock Chelsea, New York, 2 edition, 1970.

\bibitem{hybrid}
S.~Chen and S.~M. Watt.
\newblock Combinatorics of hybrid sets.
\newblock In {\em 2016 18th International Symposium on Symbolic and Numeric
  Algorithms for Scientific Computing (SYNASC)}, pages 60--64, 2016.

\bibitem{note}
E.~R. Dolores-Cuenca and J.~L. Mendoza-Cortes.
\newblock Computing order series/ehrhart polynomials of posets with
  mathematica.
\newblock In {\em the Notebook Archive}, 2022.
\newblock \url{https://notebookarchive.org/2022-02-3pvm73a}.

\bibitem{euler}
L.~Euler.
\newblock {\em Institutiones Calculi differentialis}.
\newblock Acad. Imperialis Sci., Petrograd, 1755.

\bibitem{infzetaoddnew}
S.~Fischler, J.~Sprang, and W.~Zudilin.
\newblock Many odd zeta values are irrational.
\newblock {\em Compositio Mathematica}, 155(5):938–952, 2019.

\bibitem{expansions}
P.~Flajolet and I.~Vardi.
\newblock Zeta function expansions of classical constants, 1992.
\newblock \url{http://algo.inria.fr/flajolet/Publications/landau.ps}.

\bibitem{lcomb}
H.~W. Gould.
\newblock {\em Combinatorial Identities: A Standardized Set of Tables Listing
  500 Binomial Coefficient Summations}.
\newblock Morgantown Printing and Biding Co, Morgantown, W. Va, 1972.

\bibitem{sums}
R.~E. Haddad.
\newblock Repeated sums and binomial coefficients.
\newblock {\em Open Journal of Discrete Applied Mathematics}, 4:30--47, 2021.

\bibitem{Doppelgangers}
Z.~Hamaker, R.~Patrias, O.~Pechenik, and N.~Williams.
\newblock {Doppelg\"angers: Bijections of Plane Partitions}.
\newblock {\em International Mathematics Research Notices}, 2020(2):487--540,
  03 2018.

\bibitem{A019538}
O.~F. Inc.
\newblock Entry a019538 in the on-line encyclopedia of integer sequences, 2022.
\newblock \url{http://oeis.org/A019538}.

\bibitem{MATLAB}
T.~M. Inc.
\newblock Matlab version: 9.14.0 (r2023a), 2023.

\bibitem{state}
L.~Lai and P.~Yu.
\newblock A note on the number of irrational odd zeta values.
\newblock {\em Compositio Mathematica}, 156(8):1699–1717, 2020.

\bibitem{atleastwo}
L.~Lai and L.~Zhou.
\newblock At least two of $\zeta(5),\zeta(7),\ldots,\zeta(35)$ are irrational,
  2021.

\bibitem{ope}
J.-L. Loday and B.~Vallette.
\newblock {\em Algebraic Operads}.
\newblock Springer-Verlag Berlin Heidelberg, 1st edition, 2012.

\bibitem{nbinomial}
D.~Loeb.
\newblock A generalization of the binomial coefficients.
\newblock {\em Discrete Math.}, 105(1-3):143--156, 1992.

\bibitem{negative}
D.~Loeb.
\newblock Sets with a negative number of elements.
\newblock {\em Adv. Math.}, 91(1):64--74, 1992.

\bibitem{EulerandStirling}
C.~Mariconda and A.~Tonolo.
\newblock {\em Discrete Calculus, Methods for Counting}.
\newblock Springer International Publishing, Switzerland, 2016.

\bibitem{atf}
M.~Markl, S.~Shnider, and J.~Stasheff.
\newblock {\em Operads in algebra, topology and physics}, volume~96 of {\em
  Math. Surv. Monogr.}
\newblock Providence, RI: American Mathematical Society (AMS), 2002.

\bibitem{trascendental}
M.~R. Murty and P.~Rath.
\newblock {\em Transcendental numbers}.
\newblock New York, NY: Springer, 2014.

\bibitem{conv}
P.~Nandy.
\newblock Conversation (interview with s, janaki ammal).

\bibitem{Series}
D.~{\v{Z}}. {\DJ}okovi{\'c}.
\newblock Summation of certain types of series.
\newblock {\em Publications de l'Institut Math\'ematique}, 4(18)(24):43--55,
  1964.

\bibitem{eulerianB}
T.~K. Petersen.
\newblock {\em Eulerian Numbers}.
\newblock Birkhäuser Advanced Texts Basler Lehrbücher Series. Birkhäuser,
  New York, NY, 1 edition, 2015.

\bibitem{infzeta}
T.~Rivoal.
\newblock There are infinitely many irrational values of the {Riemann} zeta
  function at odd integers.
\newblock {\em C. R. Acad. Sci., Paris, S{\'e}r. I, Math.}, 331(4):267--270,
  2000.

\bibitem{oddzeta}
T.~Rivoal and W.~Zudilin.
\newblock A note on odd zeta values.
\newblock {\em S{\'e}min. Lothar. Comb.}, 81:b81b, 13, 2020.

\bibitem{osets}
B.~Schr{\"o}der.
\newblock {\em Ordered sets. {An} introduction with connections from
  combinatorics to topology}.
\newblock Basel: Birkh{\"a}user/Springer, 2nd edition edition, 2016.

\bibitem{tropical}
D.~Speyer and B.~Sturmfels.
\newblock Tropical mathematics.
\newblock {\em Math. Mag.}, 82(3):163--173, 2009.

\bibitem{elementary}
J.~Sprang.
\newblock Infinitely many odd zeta values are irrational. by elementary means,
  2018.

\bibitem{sumsS}
H.~Srivastava.
\newblock Sums of certain series of the riemann zeta function.
\newblock {\em Journal of Mathematical Analysis and Applications},
  134(1):129--140, 1988.

\bibitem{beginning}
R.~P. Stanley.
\newblock A chromatic-like polynomial for ordered sets.
\newblock In {\em Proc. Second Chapel Hill Conf. on Combinatorial Mathematics
  and its Applications}, pages 421--427, May 1970.

\bibitem{crts}
R.~P. Stanley.
\newblock Combinatorial reciprocity theorems.
\newblock {\em Adv. Math.}, 14:194--253, 1974.

\bibitem{two}
R.~P. Stanley.
\newblock Two poset polytopes.
\newblock {\em Discrete \& Computational Geometry}, 1:9--23, 1986.

\bibitem{monotonicity}
R.~P. Stanley.
\newblock A monotonicity property of {{\(h\)}}-vectors and {{\(h^*\)}}-vectors.
\newblock {\em Eur. J. Comb.}, 14(3):251--258, 1993.

\bibitem{enumerative}
R.~P. Stanley.
\newblock {\em Enumerative combinatorics. {Vol}. 1.}, volume~49 of {\em Camb.
  Stud. Adv. Math.}
\newblock Cambridge: Cambridge University Press, 2nd ed. edition, 2012.

\bibitem{order}
W.~Trotter.
\newblock {\em Combinatorics and Partially Ordered Sets: Dimension Theory}.
\newblock Johns Hopkins Studies in the Mathematical Sciences. Johns Hopkins
  University Press, 2002.

\bibitem{gf}
H.~S. Wilf.
\newblock {\em generatingfunctionology}.
\newblock A K Peters Ltd., Wellesley, MA, third edition, 2006.

\bibitem{Mathematica}
I.~Wolfram~Research.
\newblock Mathematica, {V}ersion 13.2.
\newblock Champaign, IL, 2022.

\bibitem{w}
J.~Worpitzky.
\newblock Studien {\"u}ber die bernoullischen und eulerschen zahlen.
\newblock {\em J. Reine Angew. Math.}, 94:203--232, 1883.

\bibitem{z57911}
W.~{Zudilin}.
\newblock {One of the numbers {\ensuremath{\zeta}}(5), {\ensuremath{\zeta}}(7),
  {\ensuremath{\zeta}}(9), {\ensuremath{\zeta}}(11) is irrational}.
\newblock {\em Russian Mathematical Surveys}, 56(4):774--776, Aug. 2001.

\bibitem{elementaryz}
W.~Zudilin.
\newblock One of the odd zeta values from {{\(\zeta(5)\)}} to {{\(\zeta(25)\)}}
  is irrational. {By} elementary means.
\newblock {\em SIGMA, Symmetry Integrability Geom. Methods Appl.}, 14:paper
  028, 8, 2018.

\end{thebibliography}

\end{document}